\documentclass{amsart}
\usepackage{amssymb,amsmath,amsfonts}
\usepackage{verbatim}

\def\XXint#1#2#3{{\setbox0=\hbox{$#1{#2#3}{\int}$}
     \vcenter{\hbox{$#2#3$}}\kern-.5\wd0}}

\usepackage{color}

\newtheorem{theorem}{Theorem}[section]
\newtheorem{lemma}[theorem]{Lemma}
\newtheorem{corollary}[theorem]{Corollary}
\newtheorem{proposition}[theorem]{Proposition}

\theoremstyle{definition}

\theoremstyle{remark}
\newtheorem{remark}[theorem]{Remark}

\newcommand{\mysection}[1]{\section{#1}
\setcounter{equation}{0}}

\newcommand{\bR}{\mathbb R}

\newcommand{\bZ}{\mathbb Z}

\renewcommand{\epsilon}{\varepsilon}

\begin{document}
\title[quasi-geostrophic equations]
{A regularity criterion for the dissipative quasi-geostrophic
equations}

\author[H. Dong]{Hongjie Dong}
\address[H. Dong]
{The Division of Applied Mathematics, Brown University,
182 George Street, Box F, Providence, RI 02912}
\email{hdong@brown.edu}

\author[N. Pavlovi\'{c}]{Nata\v{s}a Pavlovi\'{c}}
\address[N. Pavlovi\'{c}]
{Department of Mathematics, University of Texas at Austin,
1 University Station, C1200, Austin, Texas 78712}
\email{natasa@math.utexas.edu}

\thanks{H.D. was partially supported by a start-up funding from the Division
of Applied Mathematics of Brown University. N.P. was partially supported by a start-up funding
from the College of Natural Sciences of the University of Texas at Austin.}



\begin{abstract}
We establish a regularity criterion for weak solutions of the
dissipative quasi-geostrophic equations in mixed time-space Besov
spaces.
\end{abstract}

\maketitle

\mysection{Introduction}
                                                    \label{intro}
In this paper we obtain a regularity criterion for weak solutions of
the 2D dissipative quasi-geostrophic equations.
We consider the following initial value problem
\begin{equation} \label{qgeq1}
\left\{\begin{array}{l l}
\theta_t+u\cdot \nabla\theta+(-\Delta)^{\gamma/2}\theta=0,
\quad x \in \bR^2, t \in (0,\infty),\\
\theta(0,x)=\theta_0(x),\end{array}\right.
\end{equation}
where $\gamma\in (0,2]$ is a fixed parameter and the velocity
$u=(u_1,u_2)$ is divergence free and determined by the Riesz
transforms of the potential temperature $\theta$:
$$
u=(-{\mathcal R}_2\theta,{\mathcal R}_1\theta)=(-\partial_{x_2}
(-\Delta)^{-1/2}\theta,
\partial_{x_1}(-\Delta)^{-1/2}\theta).
$$

The 2D quasi-geostrophic equation is an important model in
geophysical fluid dynamics used in meteorology and oceanography
(see, for example, Pedlosky \cite{Pe87}).
It is derived from general quasi-geostrophic equations
in the special case of constant
potential vorticity and buoyancy frequency.

The main mathematical question concerning the initial value problem
\eqref{qgeq1} is whether there exists a global in time smooth
solution to \eqref{qgeq1} evolving from any given smooth initial
data. Before we recall the known results in this direction we note
that cases $\gamma>1$, $\gamma=1$ and $\gamma<1$ are called
subcritical, critical and supercritical, respectively. Existence of
a global weak solution was established by Resnick \cite{Re95}.
Furthermore, in the subcritical case, Constantin and Wu \cite{CW99}
proved that every sufficiently smooth initial data give a rise to a
unique global smooth solution. In the critical case, $\gamma = 1$,
Constantin, Cordoba and Wu \cite{CCW01} established existence of a
unique global classical solution corresponding to any initial data
that are small in $L^{\infty}$. The assumption requiring smallness
in $L^{\infty}$ was removed by Caffarelli and Vasseur \cite{CV06},
Kiselev, Nazarov and Volberg \cite{KNV07} and Dong and Du
\cite{dongdu}. In \cite{KNV07} the authors proved persistence of a
global solution in $C^{\infty}$ corresponding to any $C^{\infty}$
periodic initial data. Dong and Du in \cite{dongdu} adapted the
method of \cite{KNV07} and obtained global well-posedness for the
critical 2D dissipative quasi-geostrophic equations with $H^1$
initial data in the whole space. On the other hand, Caffarelli and
Vasseur established regularity of Leray-Hopf solution by proving the
following three claims:
\begin{enumerate}
\item Every Leray-Hopf weak  solution corresponding to initial data
$\theta_0 \in L^2$ is in $L^{\infty}_{\text{loc}}({\mathbb R}^2
\times (0, \infty))$
\item The $L^{\infty}$ solutions are H\"{o}lder regular i.e. they are in  $C^{\gamma}$
for some $\gamma > 0$
\item Every H\"{o}lder regular solution is a classical solution in $C^{1, \beta}$.
\end{enumerate}

While the main question addressing global in time existence is
settled in the critical case, it still remains open in the
supercritical case, $\gamma < 1$. In this case Chae and Lee
\cite{chaeLee}, Wu \cite{Wu04} and Chen, Miao and Zhang \cite{miao}
established existence of a global solution in Besov spaces evolving
from small initial data (see also \cite{miura, ju4}). Recently,
Constantin and Wu in \cite{CW06} implemented the approach of
\cite{CV06} in the supercritical case. They proved that every
Leray-Hopf weak solution corresponding to initial data $\theta_0 \in
L^2$ is in $L^{\infty}_{\text{loc}}({\mathbb R}^n \times (0,
\infty))$ and hence the claim (1) is valid in the supercritical
case. Concerning an analogue of the claim (2), Constantin and Wu in
\cite{CW06} proved that $L^{\infty}$ solutions are H\"{o}lder
continuous under the additional assumption that the velocity $u \in
C^{1-\gamma}$. In a separate paper \cite{CW06-Hol} Constantin and Wu
considered the step (3) of the above approach and established a
conditional regularity result of the type: if a Leray-Hopf solution
is in the sub-critical space $L^\infty((t_0,t_1);C^\delta(\bR^2))$
for some $\delta>1-\gamma$ on the time interval $[t_0, t_1]$, then
such a solution is a classical solution on $(t_0, t_1]$.

In this paper we extend the conditional regularity result of \cite{CW06-Hol} to
scaling invariant mixed time-space Besov spaces
$L^{r_0}((0, T); B_{p, \infty}^{\alpha})$ with
\begin{equation} \label{alp}
\alpha = \frac{2}{p} + 1 - \gamma + \frac{\gamma}{r_0}.
\end{equation}
More precisely, we show that if
$$\theta\in L^{r_0}_t((0,T);B^{\alpha}_{p,\infty}(\bR^2))$$
is a weak solution of the 2D quasi-geostrophic equation \eqref{qgeq1},
then $\theta$ is a classical solution
of \eqref{qgeq1} in $(0,T] \times {\mathbb R}^2$. Significance of
this space is that it is a critical space, by which we mean scaling invariant
under the scaling transformation
$$ \theta_{\lambda} = \lambda^{\gamma - 1} \; \theta(\lambda x, \lambda^{\gamma} t).$$
Since the following embedding relations
$$
L^\infty_tL^2_x\cap L^\infty_tC^\delta_x
\hookrightarrow
L^\infty_tL^2_x\cap L_t^\infty \dot B_{p,\infty}^{\delta(1-\frac 2 p)}
\hookrightarrow
L_t^{r_0}B_{p,\infty}^{\alpha},
$$
hold for sufficiently large $p$ and $r_0$,
our regularity result can be understood as an extension of
the regularity result of Constantin and Wu \cite{CW06-Hol}
to critical spaces.

In order to prove the regularity result we first establish local
existence and uniqueness of weak solutions to \eqref{qgeq1} in
certain mixed time-space Besov spaces of Chemin type ${\tilde L}^r
B^{\alpha}_{p,q}$ (for a definition of this space, see Section
\ref{notation}). We prove such existence and uniqueness results
following the approach of Q. Chen et al \cite{miao}. We choose
$\alpha$ according to \eqref{alp} which in turn implies that the
space $B^{\alpha}_{p,q}$ itself is subcritical. Therefore the time
of existence depends only on the norm of the initial data and not on
the profile. We combine the local existence (stated in Proposition
\ref{prop-lwp}) and uniqueness of weak solutions (stated in
Proposition \ref{prop-unique}) to prove regularity by using a
contradiction argument in the spirit of the work of Giga \cite{Gi86}
in the context of the Navier-Stokes equations.

We recall that the first conditional regularity result for solutions
to \eqref{qgeq1} was obtained by Constantin, Majda and Tabak
\cite{CMT}. Recently  Chae  established a conditional regularity
result in Sobolev spaces in \cite{chae} and in Triebel-Lizorkin
spaces in \cite{chae03}, while B.-Q. Dong and Chen in \cite{DC}
extended the regularity criterion of Chae \cite{chae} to Besov
spaces by proving that a solution to \eqref{qgeq1} is regular on the
time interval $(0,T]$ if
$$\nabla \theta  \in L^r((0,T); \dot{B}^{0}_{p, \infty}) \mbox{ with }
\frac{2}{p} + \frac{\gamma}{r} = \gamma, \; \; \frac{4}{\gamma} \leq p \leq \infty.$$
In comparison with
\cite{DC} we require less regularity for $\theta$.
We note that these conditional regularity results are in the spirit
of the conditional regularity results available for the 3D Navier-Stokes equations
e.g. \cite{La, Pr, Serrin1, ESS, CS}.

\subsection*{Organization of the paper}
The paper is organized as follows. In Section \ref{notation} we
introduce the notation that shall be used throughout the paper and
we review known estimates on the nonlinear term. In Section
\ref{main} we state the main results of the paper. Then in Section
\ref{proofs} we give proof of the existence and regularity results,
while in the appendix Section \ref{apen} we fill out details of the
existence result stated in Section \ref{main}.

\mysection{Notation and preliminaries} \label{notation}

\subsection{Notation and spaces}

We recall that for any $\beta\in \bR$ the fractional Laplacian $(-\Delta)^\beta$
is defined via its Fourier transform:
$$\widehat{(-\Delta)^\beta f}(\xi)=|\xi|^{2\beta}\hat f(\xi).$$

We note that by a weak solution to \eqref{qgeq1}
we mean $\theta(t, x)$ in $(0,\infty) \times {\mathbb R}^2$
such that
for any smooth function $\phi(t,x)$ satisfying $\phi(t,\cdot) \in \mathcal{S}$ 
for each $t$, the identity
\begin{align*}
& \int_{{\mathbb R}^2} \theta(T,\cdot) \phi(T,\cdot) \; dx
- \int_{{\mathbb R}^2} \theta(0,\cdot) \phi(0,\cdot)\; dx
- \int_{0}^{T} \int_{{\mathbb R}^2} \theta \phi_t \; dx \; dt\\
&- \int_{0}^{T} \int_{{\mathbb R}^2} u \theta \nabla\phi \; dx \; dt
+ \int_{0}^{T} \int_{{\mathbb R}^2} \theta \Lambda^\gamma \phi \; dx \; dt
=0
\end{align*}
holds for any $T>0$.

Before we recall the definition of the spaces that will be used throughout the paper, we
shall review the Littlewood-Paley decomposition. For any integer
$j$, define $\Delta_j$ to be the Littlewood-Paley projection
operator with $\Delta_j v=\phi_j*v$, where
$$
\hat\phi_j(\xi)=\hat\phi(2^{-j}\xi),\quad \hat\phi \in
C_0^\infty(\bR^2\setminus\{0\}),\quad \hat\phi\geq 0,
$$
$$
\text{supp}\,\hat \phi\subset\{\xi\in \bR^2\,|\,1/2\leq |\xi|\leq 2\},
\quad \sum_{j\in \bZ}\hat\phi_j(\xi)=1\,\,\text{for}\,\,\xi\neq 0.
$$
Formally, we have the Littlewood-Paley decomposition
$$
v(\cdot,t)=\sum_{j\in \mathbb{Z}}\Delta_j v(\cdot,t).
$$
Also denote
\begin{equation*}
\Lambda=(-\Delta)^{1/2},\quad \bar\Delta_{-1}=\sum_{j<0}\Delta_j.
\end{equation*}

As usual, for any $p\in [1,\infty)$ and $s\geq 0$, we denote by $\dot
W^{s}_p$ and $W^{s}_p$, respectively the homogeneous and inhomogeneous
Sobolev spaces with norms
\begin{align*}
\|v\|_{\dot W^{s}_p}:=&\Big\|(\sum_{k\in \mathbb{Z}}|2^{ks}\Delta_k
v|^2)^{1/2}\Big\|_{L^p}\sim \|\Lambda^s v\|_{L^p},\\
\|v\|_{W^{s}_p}:= &\|v\|_{\dot W^{s}_p}+\|v\|_{L^p}.
\end{align*}
When $p=2$, we use $\dot H^s$ and $H^s$ instead of $\dot W^{s}_p$
and $W^{s}_p$. For any $p,q\in [1,\infty]$ and $s\in  \bR$, we
denote by $\dot B^{s}_{p,q}$ and $B^{s}_{p,q}$, respectively the homogeneous
and inhomogeneous Besov spaces equipped with norms
\begin{align*} \|v\|_{\dot
B^{s}_{p,q}}:=&\left\{\begin{array}{l l}
\Big(\sum_{j\in \bZ} 2^{jsq}\|\Delta_j v \|_{L^p}^q\Big)^{1/q},\quad\text{for}\,\,q<\infty,\\
\sup_{j\in \bZ}2^{js}\|\Delta_j v \|_{L^p},\quad\text{for}\,\,q=\infty,
\end{array}\right.\\
\|v\|_{B^{s}_{p,q}}:=&\left\{\begin{array}{l l} \Big(\sum_{j\geq 0}
2^{jsq}\|\Delta_j v\|_{L^p}^q\Big)^{1/q}+
\|{\bar\Delta_{-1}}v\|_{L^p},\quad\text{for}\,\,q<\infty,\\
\sup_{j\geq
0}2^{js}\|\Delta_j v \|_{L^p}+\|{\bar\Delta_{-1}}v\|_{L^p}.\quad\text{for}\,\,q=\infty,
\end{array}\right.
\end{align*}
If $s>0$, we have
$$
B^{s}_{p,q}=\dot B^{s}_{p,q}\cap L^p,\quad \|v\|_{B^{s}_{p,q}}\sim \|v\|_{\dot B^{s}_{p,q}}+\|v\|_{L^p}.
$$

For $s\in \bR$, $1\le p, q,r\le \infty$, $I$ an interval in $\bR$, the homogeneous mixed
time-space Besov space $\tilde L^r(I;\dot B^{s}_{p,q})$ is the space of distributions
in $\mathcal D(I;  \mathcal S_0^\prime(\bR^d))$ such that
\begin{align*}
\| f\|_{\widetilde{L}^r(I;\dot B^{s}_{p,q})} := \left \| 2^{sj} \left( \int_I \| \Delta_j f(t) \|_{L^p(\bR^d)}^r dt
\right)^{1/r} \right\|_{l^q(\bZ)} < \infty,
\end{align*}
(usual modification applied if $r=\infty$ or $q=\infty$). Also the
inhomogeneous time-space Besov norm is given by
\begin{align*}
\| f\|_{\widetilde L^r(I;B^{s}_{p,q})}
: = \| f\|_{L^r(I;L^p(\bR^d))} + \| f\|_{\tilde L^r(I;\dot B^{s}_{p,q})}.
\end{align*}
These spaces were introduced by J.-Y. Chemin \cite{Chemin}.

\subsection{Preliminaries}
The following Bernstein's inequality is well-known.
\begin{lemma}
                                        \label{bern}
i) Let $p\in [1,\infty]$ and $s\in \bR$. Then for any $j\in \bZ$, we
have
\begin{equation}
                                        \label{eq5.43}
\lambda 2^{js} \|\Delta_j v\|_{L^p}\leq \|\Lambda^s\Delta_j
v\|_{L^p}\leq \lambda' 2^{js} \|\Delta_j v\|_{L^p}
\end{equation}
with  some constants $\lambda$ and $\lambda'$ depending only on $p$
and $s$.

ii) Moreover, for $1\leq p\leq q\leq \infty$, there exists a
positive constant $C$ depending only on $p$ and $q$ such that
\begin{equation}
                                        \label{eq5.48}
\|\Delta_j v\|_{L^q}\leq C2^{(1/p-1/q)dj}\|\Delta_j v\|_{L^p}.
\end{equation}
\end{lemma}

Now we recall the generalized Bernstein's inequality and a lower bound
for an integral involving fractional Laplacian which will be used in the paper. They
can be found in \cite{wu4}, \cite{ju4} and \cite{miao}.
\begin{lemma}
                \label{lowerbd}
i) Let $p\in [2,\infty)$ and $\gamma\in [0,2]$. Then for any $j\in \bZ$, we
have
\begin{equation}
                                        \label{eq2.5.43}
\lambda 2^{\gamma j/p} \|\Delta_j v\|_{L^p}\leq
\|\Lambda^{\gamma/2}(|\Delta_j v|^{p/2})\|_{L^2}^{2/p}\leq \lambda'
2^{\gamma j/p} \|\Delta_j v\|_{L^p}
\end{equation}
with  some positive constants $\lambda$ and $\lambda'$ depending only on $p$ and $\gamma$.

ii) Moreover, we have
\begin{equation}
                \label{eq2.2.37}
\int_{\bR^2}(\Lambda^\gamma v)|v|^{p-2} v\geq c\|\Lambda^{\gamma/2}|v|^{p/2}\|_{L^2}^2,
\end{equation}
and
\begin{equation}
                \label{eq2.2.36}
\int_{\bR^2}(\Lambda^\gamma\Delta_j v)|\Delta_j v|^{p-2}\Delta_j v\geq c2^{\gamma j}\|\Delta_j v\|_{L^p}^p,
\end{equation} with  some positive constant $c$ depending only on $p$ and $\gamma$.
\end{lemma}

Next we recall the commutator estimate that shall be used throughout the paper.
\begin{lemma}
                        \label{comm1}
Let $d\geq 1$ be an integer, $p,q\in [1,\infty]$,
$\frac 1 r=\frac 1 {r_1}+\frac 1 {r_2}\leq 1$, $\rho_1<1$, $\rho_2<1$ and $u$ be a
divergence free vector field. Assume in addition that
$$
\rho_1+\rho_2+d\min(1,\frac 2 p)>0,\quad \rho_1+\frac d p>0
$$
Then for any $j\in \bZ$ we have
\begin{align}
            \label{eq2.4.08}
&\|[u,\Delta_j]\cdot \nabla v\|_{L_t^r(L^p(\bR^d))}\nonumber\\
&\,\, \leq Cc_j 2^{-j(\frac d p+\rho_1+\rho_2-1)}\|\nabla
u\|_{\tilde L_t^{r_1}(\dot B_{p,q}^{\frac d p+\rho_1-1}(\bR^d))}
\|\nabla v\|_{\tilde L_t^{r_2}(\dot B_{p,q}^{\frac d p+\rho_2-1}(\bR^d))},
\end{align}
where $C$ is a positive constant independent of $j$ and $\{c_j\}\in l^q$ satisfying $\|c_j\|_{l^q}\leq 1$. Here
$$
[u,\Delta_j]\cdot \nabla v=u\cdot \Delta_j(\nabla v)-\Delta_j(u\cdot \nabla v).
$$
\end{lemma}
\begin{proof}
See \cite{miao} and \cite{danchin}.
\end{proof}

Also we state the following result about a product of two functions in Besov spaces. For
a proof, see, for example, \cite{miao}.

\begin{lemma} \label{prod}
Let $s > -\frac{d}{p} - 1$, $s < s_1 < \frac{d}{p}$, $2 \leq p \leq \infty$,
$1 \leq q \leq \infty$,
$\frac{1}{r} = \frac{1}{r_1} + \frac{1}{r_2} \leq 1$ and $u$ be a divergence free vector field.
Then
$$
\|u \cdot \nabla v\|_{\widetilde{L}^r_t (\dot{B}^s_{p,q})}\lesssim
\|u\|_{\widetilde{L}^{r_1}_t (\dot{B}^{s_1}_{p,q})} \; \|\nabla
v\|_{\widetilde{L}^{r_2}_t (\dot{B}^{s+ \frac{d}{p} - s_1}_{p,q})}.
$$
If $s_1 = \frac{d}{p}$ or $s_1 = s$, then $q$ has to be taken to be $1$.
\end{lemma}

\mysection{Formulation of results}\label{main}

In this section we formulate existence and uniqueness results
that shall be used in the proof of our main regularity result.
Also we formulate the main regularity result.

First we state the local well-posedness result for
\eqref{qgeq1}.

\begin{proposition} \label{prop-lwp}
Let $\gamma\in (0,1]$, $p\in [2,\infty)$, $q\in [1,\infty]$ and
$r_0\in[2,\infty)$. Denote by $\alpha = \frac{2}{p}+1-\gamma +
\frac{\gamma}{r_0}$. Assume $\theta_0\in B^{\alpha}_{p,q}(\bR^2)$.
Then there exists $T\geq c\|\theta_0\|_{{\dot
B^{\alpha}_{p,q}}}^{-r_0}$ for some constant $c>0$ such that the
initial value problem for \eqref{qgeq1} has a unique weak solution
\begin{equation*}
\theta(t,x)\in
\widetilde{L}^2((0,T);B^{\alpha+\frac \gamma 2}_{p,q})\cap
\widetilde{L}^\infty((0,T);B^{\alpha}_{p,q}).
\end{equation*}
For any $r\in [2,\infty]$,
\begin{equation} \label{cons-apriori}
\|\theta\|_{\widetilde{L}^r_tB^{\alpha+\frac{\gamma}{r}}_{p,q}((0,T) \times \bR^2)}
\leq C\|\theta_0\|_{{B^{\alpha}_{p,q}}}
\end{equation}
with a positive constant $C$ independent of $r$, and $\theta$ is
smooth in $(0,T)\times \bR^2$. Moreover, if $q<\infty$, we also have
$$
\theta(t,x)\in C([0,T); B^{\alpha}_{p,q}).$$
\end{proposition}
\begin{remark}
From the proof, it is clear that if $r_0>2$ then the unique solution
$\theta$ is actually in
$$
\widetilde{L}^1((0,T);B^{\alpha+\gamma}_{p,q})\cap
\widetilde{L}^\infty((0,T);B^{\alpha}_{p,q}).
$$ Moreover, for any $r\in [1,\infty]$ estimate \eqref{cons-apriori}
holds. However, we will not use this in our main theorem.
\end{remark}

An analogous local well-posedness result in the critical space
$B^{\frac{2}{p} + 1 - \gamma}_{p,q}$ was established in \cite{miao}
(see also \cite{miura, ju4} for local well-posedness results in
Sobolev spaces). However, we remark that with $\theta_0$ in the
critical space the time of existence $T$ depends on the profile of
$\theta_0$ instead of the norm.

The next proposition is about the uniqueness of weak solutions
in mixed time-space Besov spaces.

\begin{proposition}\label{prop-unique}
Let $\gamma\in (0,1]$, $p\in [2,\infty)$,  $T\in (0,\infty)$  and
$r_0\in[2,\infty)$. Denote by $\alpha = \frac{2}{p}+1-\gamma +
\frac{\gamma}{r_0}$.
\begin{enumerate}
\item[(a)] Let $q\in [1,\infty)$.
If $\theta, \theta' \in \widetilde{L}^{r_0}_tB^{\alpha}_{p,q}((0,T) \times \bR^2)$
are two weak solutions of \eqref{qgeq1} with the same initial data, then $\theta =\theta'$ in $[0,T)\times\bR^2$.

\item[(b)] Let $q=\infty$.
If $\theta, \theta' \in {L}^{r_0}_tB^{\alpha}_{p,q}((0,T) \times \bR^2)$
are two weak solutions of \eqref{qgeq1} with the same initial data, then
$\theta =\theta'$ in $[0,T)\times\bR^2$.
\end{enumerate}
\end{proposition}

The following regularity criteria is our main result. Roughly
speaking, it says weak solutions in certain critical time-space
Besov spaces are regular.

\begin{theorem}
        \label{thm1}
Let $\gamma\in (0,1]$, $p\in [2,\infty)$, $T\in (0,\infty)$ and
$r_0\in[2,\infty)$. Denote by $\alpha = \frac{2}{p}+1-\gamma +
\frac{\gamma}{r_0}$. If
$$\theta\in L^{r_0}_t((0,T);B^{\alpha}_{p,\infty}(\bR^2))$$
is a weak solution of \eqref{qgeq1}, then $\theta$ is in
$C^{\infty}((0,T] \times \bR^2)$, and thus it is a classical
solution of \eqref{qgeq1} in the region $(0,T] \times \bR^2$.
\end{theorem}

\mysection{Proofs of existence, uniqueness and regularity} \label{proofs}

In this section we present proofs of the above stated results. In
order to prove Proposition \ref{prop-lwp} and Proposition
\ref{prop-unique} we modify accordingly the approach used by Q. Chen
et al \cite{miao}.

\subsection{Proof of Proposition \ref{prop-lwp}}
\subsubsection{A priori estimate}

We apply the operator $\Delta_j$ to the first equation in
\eqref{qgeq1} to obtain
\begin{equation}
                        \label{Eq5.06}
\partial_t \Delta_j \theta +\Delta_j(u \cdot \nabla \theta)
+ \Lambda^\gamma \Delta_j \theta =0,
\end{equation}
which is equivalent to
\begin{equation} \label{en-j}
\partial_t \Delta_j \theta +u \cdot \nabla \Delta_j \theta
+ \Lambda^\gamma \Delta_j \theta
=[u,\Delta_j]\cdot \nabla \theta.
\end{equation}
Now we multiply \eqref{en-j} by $|\Delta_j \theta|^{p-2}\Delta_j \theta$ and integrate in $x$.
Since $u$ is divergence free, the integration by parts yields
$$
\int_{\bR^2} u \cdot \nabla \Delta_j \theta |\Delta_j \theta|^{p-2} \Delta_j \theta
\, dx = 0.
$$
Hence we have
\begin{equation} \label{afterint}
\frac 1 p\frac d {dt}\|\Delta_j \theta\|_{L^p}^p
+ \int_{\bR^2}(\Lambda^{\gamma}\Delta_j \theta)\;|\Delta_j \theta|^{p-2}\Delta_j \theta\,dx
=\int_{\bR^2}[u,\Delta_j]\cdot \nabla \theta|\Delta_j \theta|^{p-2}\Delta_j \theta\,dx.
\end{equation}
Now we use Lemma \ref{lowerbd} to obtain a lower bound on the second term on the left
hand side of \eqref{afterint} and H\"older's inequality to get an upper bound on  the right
hand side of \eqref{afterint} to derive
\begin{equation} \label{afterlowbd}
\frac d {dt}\|\Delta_j \theta\|_{L^p}
+ \lambda 2^{\gamma j} \|\Delta_j \theta\|_{L^p}
\leq C\|[u,\Delta_j]\cdot \nabla \theta\|_{L^p},
\end{equation}
where $\lambda = \lambda(p, \gamma)>0$.
Gronwall's inequality applied on \eqref{afterlowbd} implies
\begin{equation} \label{afterG}
\|\Delta_j \theta\|_{L^p}\leq
e^{-  \lambda 2^{\gamma j}t}\|\Delta_j \theta(0)\|_{L^p}
+C\int_{0}^t  e^{-   \lambda 2^{\gamma j}
(t-s)}\|([u,\Delta_j]\cdot \nabla \theta)(s)\|_{L^p}\,ds.
\end{equation}
Fix $r \in [2, \infty]$. We take the $L^{r}_t$ norm over the
interval of time $(0,T)$ to obtain:
\begin{equation} \label{I's}
\|\Delta_j \theta\|_{L^r_t L^p_x ((0,T) \times {\bR}^2)} \leq I_1 + I_2,
\end{equation}
where
\begin{align*}
I_1 & = \| e^{-\lambda 2^{\gamma j} t}\|_{L^r_t(0,T)} \; \|\Delta_j \theta(0)\|_{L^p_x} \\
I_2 & = \left\| \int_{0}^t  e^{-   \lambda 2^{\gamma j}
(t-s)}\|([u,\Delta_j]\cdot \nabla \theta)(s)\|_{L^p_x}\,ds \right\|_{L^r_t(0,T)}.
\end{align*}

Since
$$
\| e^{-\lambda 2^{\gamma j} t}\|_{L^r_t(0,T)}
\lesssim \left(
\frac{ 1 - e^{-r \lambda 2^{\gamma j}T} }{ r \lambda 2^{\gamma j} }
\right)^{\frac{1}{r}}
\lesssim \lambda^{-\frac{1}{r}} \;2^{-\frac{\gamma}{r}j},
$$
we can bound $I_1$ from above as follows
\begin{equation} \label{I1}
I_1 \lesssim \lambda^{-\frac{1}{r}} \;2^{-\frac{\gamma}{r}j} \;
\|\Delta_j \theta(0)\|_{L^p_x}.
\end{equation}
In order to estimate $I_2$ we use Young's inequality to obtain
\begin{equation} \label{I2Y}
I_2 \lesssim \| e^{-\lambda 2^{\gamma j} t}\|_{L^{1}_t(0,T)} \;
\|[u,\Delta_j]\cdot \nabla \theta\|_{L^{r}_tL^p_x((0,T)\times
\bR^2)}.
\end{equation}
Since
$$ \frac{ 1 - e^{-\lambda 2^{\gamma j}T} }{ \lambda 2^{\gamma j} }
\lesssim 2^{- \gamma j},$$
as well as
$$ \frac{ 1 - e^{-\lambda 2^{\gamma j}T} }{\lambda 2^{\gamma j} }
\lesssim T,$$
we have
$$
\frac{ 1 - e^{-\lambda 2^{\gamma j}T} }{ \lambda 2^{\gamma j} } \leq
2^{-\frac{\gamma}{r_3} j} \; T^{1- \frac{1}{r_3}},$$ where $r_3$ is
arbitrary real number such that $r_3 >1$ and will be chosen later.
Hence \eqref{I2Y} implies
\begin{equation} \label{I2}
I_2 \lesssim 2^{-\frac{\gamma}{r_3} j} \; T^{1 - \frac{1}{r_3}} \;
\|[u,\Delta_j]\cdot \nabla \theta\|_{L^{r}_t L^p_x((0,T)\times
\bR^2)}.
\end{equation}

Now \eqref{I's} combined with \eqref{I1} and \eqref{I2} gives
\begin{align}
&\|\Delta_j \theta\|_{L^r_t L^p_x ((0,T) \times {\bR}^2)}\nonumber\\
&\quad\quad \lesssim \lambda^{-\frac{1}{r}} \;2^{-\frac{\gamma}{r}j}
\; \|\Delta_j \theta(0)\|_{L^p_x} + 2^{-\frac{\gamma}{r_3} j} \;
T^{1- \frac{1}{r_3}} \; \|[u,\Delta_j]\cdot \nabla \theta\|_{L^{r}_t
L^p_x((0,T)\times \bR^2)}.   \label{Lr}
\end{align}
After we multiply \eqref{Lr} by $2^{(\alpha + \frac{\gamma}{r})j}$ and take $l^q(\bZ)$ norm
we infer:
\begin{equation} \label{lq}
\|\theta\|_{ \widetilde{L}^r(\dot{B}_{p,q}^{ \alpha +
\frac{\gamma}{r} }) } \lesssim \lambda^{-\frac{1}{r}}
\|\theta(0)\|_{\dot{B}_{p,q}^{\alpha}} +  T^{1- \frac{1}{r_3}} \; \|
2^{(-\frac{\gamma}{r_3} + \alpha + \frac{\gamma}{r})j}
[u,\Delta_j]\cdot \nabla \theta\|_{L^{r}_t L^p_x((0,T)\times \bR^2)}
\; \|_{l^q},
\end{equation}

In order to estimate $\| 2^{(-\frac{\gamma}{r_3} + \alpha +
\frac{\gamma}{r})j} [u,\Delta_j]\cdot\nabla \theta\|_{L^{r}_t
L^p((0,T)\times \bR^2)} \; \|_{l^q}$ we apply Lemma \ref{comm1} with
$$
v=\theta,\quad d=2,\quad r_1=r_2=2r, \quad \rho_1=\rho_2=1 - \gamma
+ \frac{\gamma}{2r} + \frac{\gamma}{r_0}<1
$$
and use the boundedness of the Riesz transforms to obtain
\begin{align*}
&\|[u,\Delta_j]\cdot\nabla \theta\|_{L^{r}_t L^p((0,T)\times \bR^2)} \\
& \quad \quad \lesssim c_j 2^{-( \alpha + \frac{\gamma}{r_0} +
\frac{\gamma}{r} - \gamma)j} \; \|u \|_{ \widetilde{L}^{r_1}(
\dot{B}^{\alpha + \frac{\gamma}{r_1}}_{p,q} )} \|\theta \|_{
\widetilde{L}^{r_2}(
\dot{B}^{\alpha + \frac{\gamma}{r_2}}_{p,q} )} \\
& \quad \quad \lesssim c_j 2^{-(\alpha + \frac{\gamma}{r_0} +
\frac{\gamma}{r} -\gamma)j} \; \|\theta \|_{ \widetilde{L}^{r_1}(
\dot{B}^{\alpha + \frac{\gamma}{r_1}}_{p,q} )} \|\theta \|_{
\widetilde{L}^{r_2}( \dot{B}^{\alpha + \frac{\gamma}{r_2}}_{p,q} )},
\end{align*}
where $c_j \in l^q$ is such that $\|c_j\|_{l^q} \leq 1$.
Therefore
\begin{align}
& 2^{(-\frac{\gamma}{r_3} + \alpha + \frac{\gamma}{r})j}
\|[u,\Delta_j]\cdot\nabla \theta\|_{L^{r}_t L^p_x((0,T)\times \bR^2)} \nonumber \\
& \quad \quad \lesssim c_j 2^{(-\frac{\gamma}{r_3} -
\frac{\gamma}{r_0} + \gamma)j} \|\theta \|_{ \widetilde{L}^{r_1}(
\dot{B}^{\alpha + \frac{\gamma}{r_1}}_{p,q} )} \|\theta \|_{
\widetilde{L}^{r_2}( \dot{B}^{\alpha + \frac{\gamma}{r_2}}_{p,q} )}.
\label{aftercomm}
\end{align}
After we choose $r_3$ such that
\begin{equation} \label{r3}
1 = \frac{1}{r_3} + \frac{1}{r_0},
\end{equation}
we observe that \eqref{aftercomm} implies
\begin{align}
& \| 2^{(-\frac{\gamma}{r_3} + \alpha + \frac{\gamma}{r})j}
\|[u,\Delta_j]\cdot\nabla \theta\|_{L^{r}_t L^p_x((0,T)\times \bR^2)}
\; \|_{l^q} \nonumber \\
& \quad \quad \lesssim
\|\theta \|_{ \widetilde{L}^{r_1}(
\dot{B}^{\alpha + \frac{\gamma}{r_1}}_{p,q} )}
\|\theta \|_{ \widetilde{L}^{r_2}(
\dot{B}^{\alpha + \frac{\gamma}{r_2}}_{p,q} )}. \label{afterlq}
\end{align}

Now we combine \eqref{lq} and \eqref{afterlq} together with \eqref{r3} to conclude
\begin{equation} \label{apriori}
\|\theta\|_{ \widetilde{L}^r(\dot{B}_{p,q}^{ \alpha + \frac{\gamma}{r} }) }
\lesssim
\lambda^{-\frac{1}{r}} \|\theta(0)\|_{\dot{B}_{p,q}^{\alpha}}
+  T^{\frac{1}{r_0}} \;
\|\theta \|_{ \widetilde{L}^{r_1}(
\dot{B}^{\alpha + \frac{\gamma}{r_1}}_{p,q} )}
\|\theta \|_{ \widetilde{L}^{r_2}(
\dot{B}^{\alpha + \frac{\gamma}{r_2}}_{p,q} )}.
\end{equation}
which is our main a priori estimate. In particular, if we denote by
$$
\Lambda(\theta,T)=\|\theta\|_{ \widetilde{L}^2 (\dot{B}_{p,q}^{
\alpha + \frac \gamma 2}) } +\|\theta\|_{
\widetilde{L}^\infty(\dot{B}_{p,q}^{ \alpha }) },
$$
we then have
\begin{equation}
            \label{eq11.49}
\Lambda(\theta,T) \lesssim \|\theta(0)\|_{\dot{B}_{p,q}^{\alpha}} +
T^{\frac{1}{r_0}} \;\Lambda(\theta,T)^2.
\end{equation}
With a help of the a priori estimate \eqref{apriori},  it is standard to
construct a solution of \eqref{qgeq1} by using approximations (see,
for example, \cite{miao}). For the sake of completeness, we give a
sketch of a proof  in the Appendix. We refer to \cite{dong} and \cite{dongli2}
for the proof of the smoothness of $\theta$ in $(0,T]\times \bR^2$.

\subsubsection{Uniqueness}
The proof of the uniqueness part of Proposition \ref{prop-lwp}  is
not much different from that of Proposition \ref{prop-unique}. We
refer the reader to the next section for details.

\subsection{Proof of Proposition \ref{prop-unique}}
Here we establish the uniqueness result for weak solutions to
\eqref{qgeq1}, i.e. Proposition \ref{prop-unique}.
Suppose that $\theta$ and $\theta'$ are two solutions to
\eqref{qgeq1} in $\widetilde{L}^{r_0}_tB^{\alpha}_{p,q}((0,T) \times \bR^2)$ which correspond
to the same initial data $\theta_0(x)$. We denote
$\delta \theta = \theta - \theta^{'}$ and $\delta u = u - u'$, where
$u'=(-{\mathcal R}_2\theta',{\mathcal R}_1\theta')$. Then it follows that:
\begin{equation} \label{d-qg}
\left\{\begin{array}{l l l}
\partial_t \delta \theta + u\cdot \nabla \delta \theta
+ \delta u \cdot \nabla \theta' +
\Lambda^{\gamma} \delta \theta=0, \quad x \in \bR^2, t > 0,\\
\delta u = {\mathcal R}^{\bot} \delta \theta, \\
\delta \theta (x,0)= 0. \end{array}\right.
\end{equation}

We follow the strategy used to derive \eqref{afterlowbd} to obtain
\begin{equation} \label{d-afterlowbd}
\frac d {dt}\|\Delta_j \delta \theta\|_{L^p}
+ \lambda 2^{\gamma j} \|\Delta_j \delta \theta\|_{L^p}
\leq C \left( \|[u,\Delta_j]\cdot \nabla \delta \theta\|_{L^p}
+ \| \Delta_j (\delta u \cdot \nabla \theta')\|_{L^p}\right).
\end{equation}
Since $\delta \theta (x,0) = 0$, Gronwall's inequality
applied on \eqref{d-afterlowbd} implies
$$
\|\Delta_j \delta \theta\|_{L^p}
\leq C \int_{0}^t  e^{-   \lambda 2^{\gamma j}(t-s)}
\left( \| ([u,\Delta_j]\cdot \nabla \delta \theta)(s) \|_{L^p}
+ \| (\Delta_j (\delta u \cdot \nabla \theta')(s) \|_{L^p}\right) \; ds.
$$
We take the $L^{r_0}_t$ norm over the interval of time $(0,T)$
and use Young's inequality to obtain:
\begin{align}
& \|\Delta_j \delta \theta\|_{L^{r_0}_t L^p_x ((0,T) \times {\bR}^2)} \label{d-afterY} \\
& \quad \leq C \| e^{-\lambda 2^{\gamma j} t}\|_{L^{r'}_t(0,T)} \;
\left( \|[u,\Delta_j]\cdot \nabla \delta \theta\|_{
L^{\frac{r_0}{2}}_tL^p_x((0,T)\times \bR^2) }+ \| \Delta_j (\delta u
\cdot \nabla \theta')\|_{ L^{\frac{r_0}{2}}_tL^p_x((0,T)\times
\bR^2) } \right), \nonumber
\end{align}
where $\frac{1}{r'} = 1 - \frac{1}{r_0}$.

Now let us pick $\eta$ such that
\begin{equation} \label{eta}
1 - \frac{\gamma}{r'} - \eta + \frac{4}{p} > 0.
\end{equation}
We bound $\| e^{-\lambda 2^{\gamma j} t}\|_{L^{r'}_t(0,T)}$ from above by
$2^{-\frac{\gamma}{r'}j}$, then
multiply \eqref{d-afterY} by $2^{(\frac{2}{p} - \eta)j}$ and take
take $l^q$ norm with respect to $j$ to infer:
\begin{equation} \label{d-lq}
\|\delta \theta\|_{ \widetilde{L}^{r_0}(\dot{B}_{p,q}^{ \frac{2}{p} -\eta }) }
\lesssim C(I_3 + I_4),
\end{equation}
where
\begin{align*}
I_3 & = \left\| 2^{(\frac{2}{p} - \eta - \frac{\gamma}{r'})j}
\|[u,\Delta_j]\cdot \nabla \delta \theta\|_{L^{\frac{r_0}{2}}_t L^p_x((0,T)\times \bR^2)}
\right\|_{l^q(\mathbb Z)},\\
I_4 & =  \| \delta u \cdot \nabla \theta'\|_{\widetilde{L}^{\frac{r_0}{2}}_t
\dot{B}^{\frac{2}{p} - \eta - \frac{\gamma}{r'}}_{p,q} ((0,T)\times \bR^2) }.
\end{align*}

In order to estimate $I_3$ we apply Lemma \ref{comm1} with
$$
v=\delta \theta,\quad d=2,\quad (r_1, r_2) = (r_0, r_0), \quad
(\rho_1,\rho_2) = (1 - \frac{\gamma}{r'}, - \eta)
$$
and the boundedness of the Riesz transforms as follows
\begin{align*}
&\|[u,\Delta_j]\cdot\nabla \delta \theta\|_{L^{\frac{r_0}{2}}_t L^p((0,T)\times \bR^2)} \\
& \quad \quad \lesssim c_j
2^{-( \frac{2}{p} - \frac{\gamma}{r'} - \eta)j} \;
\|u \|_{ \widetilde{L}^{r_0}(
\dot{B}^{\frac{2}{p} - \frac{\gamma}{r'} + 1}_{p,q} )} \;
\|\delta \theta \|_{ \widetilde{L}^{r_0}(
\dot{B}^{\frac{2}{p} - \eta }_{p,q} )} \\
& \quad \quad \lesssim c_j
2^{-( \frac{2}{p} - \frac{\gamma}{r'} - \eta)j} \;
\|\theta \|_{ \widetilde{L}^{r_0}(
\dot{B}^{\frac{2}{p} - \frac{\gamma}{r'} + 1}_{p,q} )} \;
\|\delta \theta \|_{ \widetilde{L}^{r_0}(
\dot{B}^{\frac{2}{p} - \eta }_{p,q} )},
\end{align*}
where $c_j \in l^q$ is such that $\|c_j\|_{l^q} \leq 1$.
Since
$$
\frac{2}{p} - \frac{\gamma}{r'} + 1 = \alpha,
$$
we obtain
\begin{equation} \label{d-I3}
I_3 \lesssim \|\theta\|_{ \widetilde{L}^{r_0}(
\dot{B}^{\alpha}_{p,q} )} \;
\|\delta \theta \|_{ \widetilde{L}^{r_0}(
\dot{B}^{\frac{2}{p} - \eta }_{p,q} )}.
\end{equation}
On the other hand to estimate $I_4$ we use Lemma \ref{prod}
with
$$ s = \frac{2}{p} - \frac{\gamma}{r'} - \eta, \quad
s_1 = \frac{2}{p} - \eta$$
and the boundedness of the Riesz transforms to obtain
\begin{equation} \label{d-I4}
I_4 \lesssim \|\delta \theta \|_{ \widetilde{L}^{r_0}(
\dot{B}^{\frac{2}{p} - \eta }_{p,q} )} \; \|\theta'\|_{
\widetilde{L}^{r_0}( \dot{B}^{\alpha}_{p,q} )}.
\end{equation}

Now we combine \eqref{d-lq}, \eqref{d-I3} and \eqref{d-I4} to conclude
\begin{equation} \label{d-uniq}
\|\delta \theta\|_{ \widetilde{L}^{r_0}(\dot{B}_{p,q}^{ \frac{2}{p} -\eta }) }
\lesssim \|\delta \theta \|_{ \widetilde{L}^{r_0}(
\dot{B}^{\frac{2}{p} - \eta }_{p,q} )} \;
\left(
 \|\theta \|_{ \widetilde{L}^{r_0}(
\dot{B}^\alpha_{p,q} )} + \|\theta'\|_{ \widetilde{L}^{r_0}(
\dot{B}^{\alpha}_{p,q} )} \right).
\end{equation}
We first look at part (a) of the proposition, i.e. the case $q<\infty$.
As $T \rightarrow 0$, the terms in the parenthesis on the right
hand side of \eqref{d-uniq} go to $0$.  For part (b), i.e. $q=\infty$, from \eqref{d-uniq} and the Minkowski's inequality we get
\begin{equation} \label{d-uniq2}
\|\delta \theta\|_{ \widetilde{L}^{r_0}(\dot{B}_{p,q}^{ \frac{2}{p} -\eta }) }
\lesssim \|\delta \theta \|_{ \widetilde{L}^{r_0}(
\dot{B}^{\frac{2}{p} - \eta }_{p,q} )} \;
\left(
 \|\theta \|_{L^{r_0}(
\dot{B}^{\alpha
}_{p,q} )} + \|\theta'\|_{ L^{r_0}( \dot{B}^{\alpha}_{p,q} )}
\right).
\end{equation}
As $T \rightarrow 0$, the terms in the parenthesis on the right
hand side of \eqref{d-uniq2} go to $0$.
Thus in both cases if $T$ is chosen small enough,
then
$\|\delta \theta\|_{ \widetilde{L}^{r_0}(\dot{B}_{p,q}^{ \frac{2}{p} -\eta })
((0,T) \times \bR^2 )} = 0$, which in turn implies
$\delta \theta = 0$. Now the standard continuity argument
can be employed to show that $\delta \theta (x, t) = 0$ for all
$x \in \bR^2$ and $t \geq 0$.

\subsection{Proof of Theorem \ref{thm1}}\label{proofofthm}

We prove the theorem by a contradiction. Assume $\theta$ is not a regular solution in $(0,T)\times \bR^2$.
Without loss of generality, one may assume $T$ is the first blowup time.
Since $\theta\in {L}^{r_0}_tB^{\alpha}_{p,\infty}$, for almost all $s\in (0,T)$
we have $\theta(s,\cdot)\in B^{\alpha}_{p,\infty}$. For any such $s$,
consider the initial value problem \eqref{qgeq1} with initial data $\theta_0=\theta(s,\cdot)$.
By applying the local well-posedness result (Proposition \ref{prop-lwp}), \eqref{qgeq1} has a unique weak solution
$$
\bar\theta\in
\widetilde{L}^2((0,T_s);B^{\alpha+\frac \gamma 2}_{p,\infty})\cap
\widetilde{L}^\infty((0,T_s);B^{\alpha}_{p,\infty})\cap
\widetilde{L}^{r_0}((0,T_s);B^{\alpha+\frac \gamma {r_0}}_{p,\infty})
$$
for some
\begin{equation}
                    \label{TS}
T_s\geq c\|\theta(s,\cdot)\|_{{\dot B^{\alpha}_{p,\infty}}}^{-r_0}
\end{equation}
with a constant $c>0$ independent of $s$. Moreover, by simple
embedding relations we have
$$
\bar\theta\in \widetilde{L}^{r_0}((0,T_s);B^{\alpha+\frac \gamma
{r_0}}_{p,\infty})\hookrightarrow
\widetilde{L}^{r_0}((0,T_s);B^{\alpha}_{p,r_0}) \hookrightarrow
L^{r_0}((0,T_s);B^{\alpha}_{p,\infty}).
$$
Now we apply the uniqueness result Proposition \ref{prop-unique} and
get $\bar\theta(\cdot,\cdot)=\theta(s+\cdot,\cdot)$. The last
equality and \eqref{TS} imply that
\begin{equation*}
            \label{eq3.56}
T-s\geq c\|\theta(s,\cdot)\|_{{\dot  B^{\alpha}_{p,\infty}}}^{-r_0}.
\end{equation*}
Therefore, for almost all $s\in (0,T)$, we have
$$
\|\theta(s,\cdot)\|_{{\dot  B^{\alpha}_{p,\infty}}}\geq c^{\frac 1 {r_0}}(T-s)^{-\frac 1 {r_0}},
$$
which contradicts the condition $\theta\in L^{r_0}_t((0,T);B^{\alpha}_{p,\infty}(\bR^2))$. The theorem is proved.

\mysection{Appendix} \label{apen}

The appendix is devoted to the proof of the existence part in
Theorem \ref{prop-lwp}. Consider the following successive
approximations: $ \theta^0\equiv u^0\equiv 0$, and for
$k=0,1,2,\cdots$,
\begin{equation}
                                        \label{appro1}
\left\{\begin{array}{l l}
\theta_t^{k+1}+u^k\cdot\nabla\theta^{k+1}_x+(-\Delta)^{\gamma/2}
\theta^{k+1}=0
\quad & x \in \bR^2, t \in (0,\infty),\\
u^{k+1}=(-\mathcal R_2\theta^{k+1},\mathcal R_1\theta^{k+1})\\
\theta^{k+1}(0,x)=\theta_0(x)\quad & x\in \bR,\end{array}\right.
\end{equation}
Similar to \eqref{eq11.49}, we have
\begin{equation}
            \label{eq11.59}
\Lambda(\theta^{k+1},T) \lesssim
\|\theta(0)\|_{\dot{B}_{p,q}^{\alpha}} +  T^{\frac{1}{r_0}}
\;\Lambda(\theta^k,T)\Lambda(\theta^{k+1},T).
\end{equation}
If we choose $T=c\|\theta_0\|_{{\dot B^{\alpha}_{p,q}}}^{-r_0}$ for
small $c>0$ depending on $\lambda$ and the implicit constant in
\eqref{eq11.59}, it then holds that for any $k=0,1,2,\cdots$,
\begin{equation}
            \label{eq12.46}
\Lambda(\theta^k,T)\lesssim \|\theta_0\|_{\dot{B}_{p,q}^{\alpha}}.
\end{equation}
Due to the $L^p$ maximum principle for \eqref{qgeq1}, we also have
\begin{equation}
            \label{eq4.36}
\|\theta^k\|_{ \widetilde{L}^2 ({B}_{p,q}^{ \alpha + \frac \gamma
2}) } +\|\theta^k\|_{ \widetilde{L}^\infty({B}_{p,q}^{ \alpha
})}\lesssim \|\theta_0\|_{{B}_{p,q}^{\alpha}}.
\end{equation}

Now by the first equation of \eqref{qgeq1} and Lemma \ref{prod}, we
have $\theta^k_t\in L^\infty(B_{p,q}^{\frac \gamma {r_0}-\gamma})$
with uniformly bounded norm for $k=1,2,3,\cdots$. Since we also have
$\theta^k\in L^\infty(B_{p,q}^{\alpha})$ with uniform bounded norm,
due to Lion-Aubin compactness theorem, there exists a subsequence,
which we still denote by $\theta^k$, and $\theta=\theta(t,x)$ such
that
$$
\theta^k\to \theta \quad \text{in}\,\,L^p_{\text{loc}}((0,T)\times
\bR^2).
$$
Moreover, $\theta$ satisfies \eqref{qgeq1} in the sense of
distributions and
\begin{equation}
            \label{eq4.40}
\|\theta\|_{ \widetilde{L}^2 ({B}_{p,q}^{ \alpha + \frac \gamma 2})
} +\|\theta\|_{ \widetilde{L}^\infty({B}_{p,q}^{ \alpha })}\lesssim
\|\theta_0\|_{{B}_{p,q}^{\alpha}}.
\end{equation}
As in \cite{dongli2}, $\theta$ is smooth in $(0,T)\times \bR^2$ and
satisfies the first equation of \eqref{qgeq1} in the same region in
the classical sense.

We claim $\theta\in C([0,T);B_{p,q}^\alpha)$ if $q<\infty$. Observe
that from \eqref{Eq5.06}, Lemma \ref{prod} and Lemma \ref{bern} $i)$
we know for $j=1,2,3,\cdots$,
$$\partial_t
\Delta_j \theta\in L^\infty((0,T);B_{p,q}^\alpha).$$ It follows
immediately that
\begin{equation}
                \label{eq5.21}
\Delta_j \theta\in C([0,T);B_{p,q}^\alpha).
\end{equation}
On the other hand, \eqref{eq4.36} implies that as $k\to \infty$
$$
\sum_{|j|\leq k}\Delta_j \theta\to \theta\quad
\text{in}\,\,L^\infty((0,T);B_{p,q}^\alpha).$$ This together with
\eqref{eq5.21} proves the claim.


\end{document}